\RequirePackage{fix-cm}
\documentclass[smallcondensed]{svjour3}     
\smartqed  
\usepackage{graphicx}
\usepackage{amsmath,amssymb}
\usepackage{theoremref}
\usepackage{cite}
\usepackage{url}
\usepackage{hyperref}

\newtheorem{thm}{Theorem}
\newtheorem{cor}{Corollary}
\newtheorem{rmk}{Remark}
\newtheorem{lem}{Lemma}

\newcommand{\C}{\mathbb C}
\newcommand{\Q}{\mathbb Q}

\newcommand{\N}{\mathbb N}
\newcommand{\R}{\mathbb R}

\begin{document}

\title{Sign changes of a product of Dirichlet character and Fourier coefficients of half integral weight modular forms}


\titlerunning{Sign changes of a product}        

\author{Soufiane Mezroui}


\institute{Soufiane Mezroui \at
           LabTIC,\\
           SIC Department,\\
           ENSAT,\\
           Abdelmalek Essaadi University,\\
           Tangier, Morocco\\
           \email{mezroui.soufiane@yahoo.fr}
           }

\maketitle

\begin{abstract}
Let $f\in S_{k+1/2}(N,\chi)$ be a Hecke eigenform of half integral weight $k+1/2\,(k\geq 2)$ and the real nebentypus $\chi=\pm 1$ where the Fourier coefficients $a(n)$ are reals. We prove that the sequence $\{\chi(p^{\nu})a(tp^{2\nu})\}_{\nu\in\N}$ has infinitely many sign changes for almost all primes $p$ where $t$ is a squarefree integer such that $a(t)\neq 0$. The same result holds for the sequences of Fourier coefficients $\{a(tp^{2(2\nu+1)})\}_{\nu\in\N}$ and $\{a(tp^{4\nu})\}_{\nu\in\N}$.
\end{abstract}

\section{Introduction and statement of results}

Let $k,N\in\N$ be integers, we denote by $S_k(N,\chi)$ the space of cusp forms of weight $k$ and level $N$ with Dirichlet character $\chi\pmod N$. 

When $4|N$, we denote by $S_{k+1/2}(N,\chi)$ the space of cusp forms of half-integral weight $k+1/2$ and level $N$ with character $\chi\pmod N$. Let $S^{*}_{3/2}(N,\chi)$ be the orthogonal complement of the subspace of $S_{3/2}(N,\chi)$ generated by single-variable theta series. For $k\ge2$ we put $S^{*}_{k+1/2}(N,\chi)=S_{k+1/2}(N,\chi)$. Recall that the Shimura lift maps $S^{*}_{k+1/2}(N,\chi)$ to the space $S_{2k}(N/2,\chi^{2})$ of cusp forms of integral weight $2k$ and level $N/2$ with character $\chi^2$. 

There have been since same years many papers studying the sign changes of modular forms using Landau's theorem (see \cite{bruin,kohn1,kumar,meher,mezroui,kohn2}). The main idea is to assume that the particular sequence of Fourier coefficients of modular forms does not have infinitely many sign changes and then the contradiction is established by Landau's theorem applied to Dirichlet series of this sequence of Fourier coefficients. For example in \cite{bruin}, Bruinier and Kohnen showed particularly that if $f$ is an eigenform with half integral weight and real Fourier coefficients $a(n)$, then for all but finitely many primes $p$, the sequence $(a(t p^{2\nu}))_{\nu\in\N}$ has infinitely many sign changes with $t$ is a square free natural number such that $a(t)\neq 0$. In this work we study the problem of sign changes of those sequences when the Fourier coefficients are complex numbers and our first main theorem is the following.

\begin{thm}\thlabel{thm1}
Let $f\in S^{*}_{k+1/2}(N,\chi)$ be a Hecke eigenform of all Hecke operators $T(p^{2})$. Let
$$
f(z)=\sum_{n\ge1}a(n)e(nz),
$$
be the Fourier expansion of $f$ at $\infty$. Let $t$ be a square free natural number such that $a(t)\neq 0$. Then for all but finitely many primes $p$ with $(p,N)=1$, the sequence $\{\frac{a(tp^{2 \nu})}{\chi(p^{\nu})}\}_{\nu\in\N}$ has infinitely many sign changes.
\end{thm}

When the character $\chi=\pm 1$ is real, we obtain the following result.

\begin{cor}
Let $f\in S^{*}_{k+1/2}(N,\chi)$ be a Hecke eigenform of all Hecke operators $T(p^{2})$ with $\chi=\pm 1$. Let
$$
f(z)=\sum_{n\ge1}a(n)e(nz),
$$
be the Fourier expansion of $f$ at $\infty$. Let $t$ be a square free natural number such that $a(t)\neq 0$. Then for all but finitely many primes $p$ with $(p,N)=1$, the sequence $\{\chi(p^{\nu})a(tp^{2\nu})\}_{\nu\in\N}$ has infinitely many sign changes.
\end{cor}

Notice that when $\chi=1$, this result is the same as \cite[Theorem 2.2]{bruin}. Further, it has been shown in \cite{bruin} that if $\chi=\pm 1$ and the Fourier coefficients $a(n)$ of the Hecke eigenform $f\in S^{*}_{k+1/2}(N,\chi)$ are reals, then the sequence $(a(t p^{2 m}))_{m\in\mathbb{N}}$ has infinitely many sign changes for almost all primes $p$. Our second main theorem shows that the subsequences of $(a(t p^{2 m}))_{m\in\mathbb{N}}$ with odd and even indices has infinitely many sign changes for almost all primes $p$.

\begin{thm}\thlabel{thm2}
Let $f\in S^{*}_{k+1/2}(N,\chi)$ be a Hecke eigenform of half integral weight $k+1/2$ and level $N$ with Dirichlet character $\chi=\pm 1$. Let
$$
f(z)=\sum_{n\ge1}a(n)e(nz),
$$
be the Fourier expansion of $f$ at $\infty$. Let $t$ be a square free natural number. Then for all but finitely many primes $p$ with $(p,N)=1$, the sequence $\{a(tp^{2(2\nu+1)})\}_{\nu\in\N}$ has infinitely many sign changes. The same result holds for the sequence $\{a(tp^{4\nu})\}_{\nu\in\N}$.
\end{thm}

Applying Landau's theorem, we will show further that the subsequences over arithmetic progressions also has infinitely many sign changes.

\begin{thm}\thlabel{thm3}
Let $f\in S^{*}_{k+1/2}(N,\chi)$ be a Hecke eigenform of all Hecke operators $T(p^{2})$. Let
$$
f(z)=\sum_{n\ge1}a(n)e(nz),
$$
be the Fourier expansion of $f$ at $\infty$. Let $t$ be a square free natural number such that $a(t)\neq 0$. Let $p,q$ be two primes. Consider all integers $d$ and $n$ such that $d$ is the smallest integer satisfying $p^{d}\equiv h\pmod{q}$ and $n>0$ is the smallest integer for which $p^{n}\equiv 1\pmod{q}$ with $h$ runs through the integers satisfying $1<h< q$.  Then for all but finitely many primes $p$ satisfying those conditions, the sequence $\{\frac{a(tp^{2(n\nu+d)})}{\chi(p)^{d+n\nu}}\}_{\nu\in\N}$ has infinitely many sign changes. 
\end{thm}

Finally, it should be noted that the main idea of this work and \cite{mezroui} is to divide by an appropriate character to obtain a results about sign changes of complex Fourier coefficients. The idea may be used to extend the results of \cite{kohn2,meher}.  

\section{Preliminary lemmas}

Let $\mathrm{Sh}_t(f)$ the Shimura lift of $f$ with respect to $t$ and let $\lambda_p$ denote the $p$-th Hecke eigenvalue of $f$. Denote by $A_{t}(n)$ the Fourier coefficients of $\mathrm{Sh}_t(f)$. Since 
$$
T(p)\mathrm{Sh}_t(f)=\mathrm{Sh}_t(T(p^2)f),
$$
then the $p$-th Hecke eigenvalue of $\mathrm{Sh}_t(f)$ is $\lambda_p$, where $T(p^2)$ is the Hecke operator on $S_{k+1/2}(N,\chi)$ and $T(p)$ is the Hecke operator on $S_{2k}(N/2,\chi^2)$. Recall that if $p$ is a prime, the Fourier coefficients of $f$ and its Shimura lift $\mathrm{Sh}_t(f)$ are related by  
\begin{equation}\label{eqtt}
a(t p^{2})=A_{t}(p)-\chi_{t,N}(p)p^{k-1}
\end{equation}
where $\chi_{t,N}(.)=\chi(.)\chi_{1}(.)$, $\chi_{1}(.)=\left(\frac{(-1)^{k}N^{2}t}{.}\right)$. 

The following lemmas will be useful.

\begin{lem}
Let $n$ be an integer such that $(n,N)=1$. Then $\frac{a(t n^{2})}{\chi(n)}\in\R$.
\end{lem}

\begin{proof}
This will be deduced by induction. Indeed, let $p\nmid N$ be a prime, the formula \eqref{eqtt} gives 
$$
a(t p^{2})=A_{t}(p)-\chi_{t,N}(p)p^{k-1}=A_{t}(p)-\chi(p)\left(\frac{(-1)^{k}N^{2}t}{p}\right)p^{k-1}.
$$
Since $\frac{A_{t}(p)}{\chi(p)}\in\R$, then $\frac{a(t p^{2})}{\chi(p)}\in\R$.
  
Since $t$ is squarefree, by \cite[pp. 452]{shim73} we have 

\begin{align}
\lambda _{p}a(t)=&a(p^{2}t)+\chi_{t,N}(p)p^{k-1}a(t),\\
\lambda _{p}a(p^{2 m}t)=&a(p^{2m+2}t)+\chi(p)^{2}p^{2k-1}a(p^{2m-2}t), (m>0).
\end{align}

The first equation gives $\frac{\lambda _{p}}{\chi(p)}\in\R$, and the second equation yields
$$
\frac{\lambda _{p}}{\chi(p)}\frac{a(p^{2 m}t)}{\chi(p^{m})}=\frac{a(p^{2m+2}t)}{\chi(p^{m+1})}+p^{2k-1}\frac{a(p^{2m-2}t)}{\chi(p^{m-1})}, (m>0).
$$
We deduce by induction that $\forall m\in N$, $\frac{a(p^{2m}t)}{\chi(p^{m})}\in\R$. Since $\forall m,n\in\mathbb{N}$ such that $(m,n)=1$,
$$
a(t m^{2})a(t n^{2})=a(t)a(t m^{2}n^{2}),
$$
the lemma follows by induction.
\end{proof}

\begin{lem}\thlabel{lem1}
Let $n\in\mathbb{N}$ be an integer and $p$ a prime such that $p\nmid n$. Let $H_{n}(X)$ be the sum
$$
H_{n}(X)=\sum_{m=0}^{\infty}\frac{a(t p^{2 m}n^{2})}{\chi(p^{m}n)}X^{m}.
$$
We have
$$
H_{n}(X)=\frac{a(t n^{2})}{\chi(n)}\left(\frac{1-\chi_{1}(p)p^{k-1}X}{1-\frac{\lambda_{p}}{\chi(p)}X+p^{2k-1}X^{2}}\right)\cdot
$$
\end{lem}

\begin{proof}

From \cite[Corolary 1.8]{shim73} we have 

\begin{gather}
\left(\frac{\lambda_{p}}{\chi(p)}\right)\left(\frac{a(t n^{2})}{\chi(n)}\right)X=\frac{a(t p^{2}n^{2})}{\chi(p n)}X+\chi_{1}(p)p^{k-1}\frac{a(t n^{2})}{\chi(n)}X,\label{eql1}\\
\frac{\lambda_{p}}{\chi(p)}\frac{a(t p^{2m}n^{2})}{\chi(p^{m}n)}X^{m+1}=\frac{a(t p^{2m+2}n^{2})}{\chi(p^{m+1}n)}X^{m+1}+p^{2k-1}\frac{a(t p^{2m-2}n^{2})}{\chi(p^{m-1}n)}X^{m+1},\label{eql2}
\end{gather}

for all $m\geq 1$. Adding $X$ times \eqref{eql1} and $X^{m}$ times \eqref{eql2} for all $m\geq 1$ to get  
$$
\frac{\lambda_p}{\chi(p)}H_{n}(X)X=H_{n}(X)-\frac{a(t n^{2})}{\chi(n)}+\chi_{1}(p)p^{k-1}\frac{a(t n^{2})}{\chi(n)}X+p^{2k-1}X^{2}H_{n}(X).
$$
This yields the result.
\end{proof}

\section{Proof of \texorpdfstring{\thref{thm1}}{Theorem 1}}

Suppose there are infinitely many primes $p\nmid N$ such that  $\left(\frac{a(tp^{2\nu})}{\chi(p^{\nu})}\right)_{\nu\in\N}$ does not have infinitely many sign changes. Applying Landau's theorem, then the series 
$$
\sum_{\nu\ge0}\frac{a(tp^{2\nu})}{\chi(p^{\nu})}p^{-\nu s},
$$
either converges for all $s\in\C$ or has a singularity at the real point of its line of convergence. 

By \thref{lem1} we have 

\begin{equation}
\sum_{\nu\ge0}\frac{a(tp^{2\nu})}{\chi(p^{\nu})}p^{-\nu s}=a(t)\frac{1-\chi_{1}(p)p^{k-1-s}}{1-\frac{\lambda_p}{\chi(p)} p^{-s}+p^{2k-1-2s}}\cdot\label{eq:12}
\end{equation}
The denominator of the right hand side of \eqref{eq:12} factorizes as follows 

\begin{equation*}
1-\frac{\lambda_p}{\chi(p)} p^{-s}+p^{2k-1-2s}=(1-\alpha_p p^{-s})(1-\beta_p p^{-s}),
\end{equation*}
where $\alpha_p+\beta_p=\frac{\lambda_p}{\chi(p)}$ and $\alpha_p\beta_p=p^{2k-1}$. Explicitly one has
\begin{equation}
\alpha_{p},\beta_{p}=\frac{\frac{\lambda(p)}{\chi(p)}\pm \sqrt{\left(\frac{\lambda(p)}{\chi(p)}\right)^{2}-4p^{2k-1}}}{2}\cdot\label{eq:13}
\end{equation}

Notice that the first alternative of Landau's theorem cannot holds, since the right-hand side of \eqref{eq:12} has a pole for $p^{s}=\alpha_p$ or $p^{s}=\beta_p$. Therefore the series 
$$
\sum_{\nu\ge0}\frac{a(tp^{2\nu})}{\chi(p^{\nu})}p^{-\nu s},
$$
has a singularity at the real point of its line of convergence for infinitely many primes $p\nmid N$. Consequently $\alpha_p$ or $\beta_p$ must be real. Suppose that $\alpha_p\in\R$. Using Deligne's bound we have $\left(\frac{\lambda(p)}{\chi(p)}\right)^{2}=\mid \lambda(p)\mid ^{2}\leq 4p^{2k-1}$, from which we get
\begin{equation}
\lambda(p)=\pm 2\,p^{k-\frac{1}{2}}\chi(p).\label{eq:14}
\end{equation}
Adjoining all $\chi(p)$ to the number field $K_f$, generated by the Hecke eigenvalues of $f$, and denote the resulted field by $\mathbb{K}$. From \eqref{eq:14} we get $\sqrt{p}\in\mathbb{K}$. Hence by our hypothesis, we conclude that there is an infinite sequence $p_1<p_2<p_3\dots$ of primes satisfying \eqref{eq:14}. We then have
$$
\Q(\sqrt{p_1},\sqrt{p_2},\sqrt{p_3},\dots)\subset \mathbb{K}.
$$
We obtain a contradiction as in \cite{bruin}, since $\Q(\sqrt{p_1},\sqrt{p_2},\sqrt{p_3},\dots)$ is an infinite extension. Thus the proof is complete.

\section{Proof of \texorpdfstring{\thref{thm2}}{Theorem 2}}

Assume the hypothesis of \thref{thm2}. We want  to compute the following sums
$$
S_{1}(X)=\sum_{m=0}^{\infty}\frac{a(t p^{2(2 m+1)})}{\chi(p^{2 m+1})}X^{2 m+1}\text{ and }S_{0}(X)=\sum_{m=0}^{\infty}\frac{ a(t p^{4 m})}{\chi(p^{2 m})}X^{2 m}.
$$
Replacing $n=1$ in \eqref{eql2} to get
$$
\frac{\lambda_{p}}{\chi(p)}\frac{a(t p^{2m})}{\chi(p^{m})}X^{m+1}=\frac{a(t p^{2m+2})}{\chi(p^{m+1})}X^{m+1}+p^{2k-1}\frac{a(t p^{2m-2})}{\chi(p^{m-1})}X^{m+1}\cdot
$$
Once again, replacing $m$ by $2m$ to obtain
$$
\frac{\lambda_{p}}{\chi(p)}\frac{a(t p^{4m})}{\chi(p^{2m})}X^{2m+1}=\frac{a(t p^{2(2m+1)})}{\chi(p^{2m+1})}X^{2m+1}+p^{2k-1}\frac{a(t p^{2(2m-1)})}{\chi(p^{2m-1})}X^{2m+1}\cdot
$$
Adding this equation for all $m\geq 1$ yields
$$
\frac{\lambda_{p}}{\chi(p)}X(S_{0}(X)-a(t))=S_{1}(X)-\frac{a(t p^{2})}{\chi(p)}X+p^{2k-1}S_{1}(X)X^{2}\cdot
$$
Hence we have 
$$
\frac{\lambda_{p}}{\chi(p)}X S_{0}(X)=S_{1}(X)(1+p^{2k-1}X^{2})-\frac{a(t p^{2})}{\chi(p)}X+a(t)\frac{\lambda_{p}}{\chi(p)}X\cdot
$$
Combining this with $S_{0}(X)+S_{1}(X)=H_{1}(X)=a(t)\left(\frac{1-\chi_{1}(p)p^{k-1}X}{1-\frac{\lambda_{p}}{\chi(p)}X+p^{2k-1}X^{2}}\right)$, then
\begin{equation}\label{eq,45}
S_{1}(X)=\frac{X\left(\frac{a(t p^{2})}{\chi(p)}-a(t)\chi_{1}(p)p^{3k-2}X^{2}\right)}{\left(1-\frac{\lambda_{p}}{\chi(p)}X+p^{2k-1}X^{2}\right)\left(1+\frac{\lambda_{p}}{\chi(p)}X+p^{2k-1}X^{2}\right)}.
\end{equation}

By the same reasoning we obtain
\begin{equation}
S_{0}(X)=\frac{a(t)\left(1+\left(p^{2k-1}-\frac{\lambda _p}{\chi(p)}\chi_{1}(p)p^{k-1}\right)X^{2}\right)}{\left(1-\frac{\lambda_{p}}{\chi(p)}X+p^{2k-1}X^{2}\right)\left(1+\frac{\lambda_{p}}{\chi(p)}X+p^{2k-1}X^{2}\right)}\cdot
\end{equation}

\begin{proof}
Assume the conditions of \thref{thm2}. Using \eqref{eq,45} and since $\chi(p)=\pm 1$, we have
\begin{multline}\label{eqq1}
\sum_{m=0}^{\infty}a(t p^{2(2 m+1)})\frac{1}{p^{s (2 m+1)}}=\\\pm\frac{p^{-s}\left(\frac{a(t p^{2})}{\chi(p)}-a(t)\chi_{1}(p)p^{3k-2-2s}\right)}{\left(1-\frac{\lambda_{p}}{\chi(p)}p^{-s}+p^{2k-1-2s}\right)\left(1+\frac{\lambda_{p}}{\chi(p)}p^{-s}+p^{2k-1-2s}\right)}\cdot
\end{multline}
Suppose that the sequence $(a(t p^{2(2 v+1)}))_{v\in\mathbb{N}}$ does not have infinitely many sign changes and apply once again Landau's theorem.

Suppose now that one of the denominators on the right-hand side of \eqref{eqq1} has a real zero. Then as in the proof of \thref{thm2} we will find
$$
\lambda(p)=\pm 2 p^{k-\frac{1}{2}}\chi(p).
$$
We repeat the procedure of \thref{thm1} to show that the right-hand side of \eqref{eqq1} has no real poles, and then this case of Landau's theorem is excluded.   

It remains to exclude the other case of Landau's theorem. For this purpose, notice that the denominator on the right-hand side of \eqref{eq,45} is coprime with $X$. Further, the polynomial $\frac{a(t p^{2})}{\chi(p)}-a(t)\chi_{1}(p)p^{3k-2}X^{2}$ is a nonzero polynomial of degree $2$ and the denominator of \eqref{eq,45} is a non constant polynomial of degree $4$, hence the denominator has zeros. Setting $X=p^{-s}$ to obtain a contradiction. The proof is completed with the similar way as above.

We proceed in a similar way to show that the sequence $\{a(t p^{4 v})\}_{v\in\mathbb{N}}$ has infinitely many sign changes for almost all primes $p$.

\end{proof}

\section{Proof of \texorpdfstring{\thref{thm3}}{Theorem 3}}

We shall compute the following sum
$$
\sum_{\nu=0}^{\infty}\frac{a(p^{d+n\nu})}{\chi_{0}(p^{d+n\nu })}X^{d+n\nu}.
$$
Assuming the hypothesis, notice first that an integer $m$ satisfies $p^{m}\equiv h\pmod{q}$ if and only if $m\equiv d\pmod{n}$. It follows from the orthogonality relations of Dirichlet characters that
\begin{equation}
\begin{split}
\sum_{\nu=0}^{\infty}\frac{a(p^{d+n\nu})}{\chi_{0}(p^{d+n\nu })}X^{d+n\nu}&=\frac{1}{\varphi(n)}\sum_{m=0}^{\infty}\frac{a(t p^{2m})}{\chi(p^{m})}X^{m}\left(\sum_{\epsilon}\epsilon(p^{m})\overline{\epsilon(h)}\right)\\
&=\frac{1}{\varphi(n)}\sum_{\epsilon}\overline{\epsilon(h)}\sum_{m=0}^{\infty}\frac{a(t p^{2m})}{\chi(p^{m})}X^{m}\epsilon(p^{m})\\
&=\frac{1}{\varphi(n)}\sum_{\epsilon}\overline{\epsilon(h)}H_{1}(X\epsilon(p)),\label{eq:2a}
\end{split}
\end{equation}
where the sum is taken over all Dirichlet characters modulo $n$.

\begin{proof}[Proof of \thref{thm3}]

Suppose that the sequence $\left(\frac{a(p^{d+n\nu})}{\chi_{0}(p^{d+n\nu })}\right)_{\nu \geq 0}$ does not have infinitely many sign
changes. We proceed as above and we exclude the two cases of Landau's theorem by using the equations \eqref{eq:2a} and \eqref{eq:12}.

\end{proof}

\begin{rmk}
Consider the primes $p$ for which the polynomial
$(\beta_{p}\alpha_{p}^{m_{p}}-\alpha_{p}\beta_{p}^{m_{p}})X^{m_{p}}+(\beta_{p}^{m_{p}}-\alpha_{p}^{m_{p}})X^{m_{p}-1}+(\alpha_{p}-\beta_{p})$ has no real zero, where $m_{p}$ is an integer satisfying $\chi(p)^{m_{p}}=1$. Then one can show as in \cite{mezroui} that for almost all of those primes $p$, the sequence $\left(\frac{a(t p^{2(l+m_{p}n)})}{\chi(p)^{l}}\right)_{n\in\N}$ has infinitely many sign changes with $l$ runs through the integers satisfying $1\leq l\leq m_{p}-1$.
\end{rmk}

\bibliographystyle{spmpsci}
\bibliography{mybibfile}

\end{document}